\documentclass[11pt]{amsart}
\usepackage[latin1]{inputenc}
\usepackage[english]{babel}
\usepackage{amssymb,amsfonts, amsmath, amsrefs, color, cite}
\usepackage[margin=2.7cm]{geometry}
\usepackage[latin1]{inputenc}
\usepackage[active]{srcltx}
\usepackage{pgf}
\usepackage{etex}
\usepackage{verbatim}
\usepackage{pgfkeys}
\usepackage{mathrsfs}
\renewcommand{\mathcal}{\mathscr}

\newtheorem{theorem}{Theorem}[section]
\newtheorem{proposition}[theorem]{Proposition}
\newtheorem{corollary}[theorem]{Corollary}
\newtheorem{lemma}[theorem]{Lemma}

\renewcommand{\le}         {\leqslant}
\renewcommand{\leq}         {\leqslant}
\renewcommand{\ge}         {\geqslant}
\renewcommand{\geq}         {\geqslant}

\renewcommand{\epsilon}{\varepsilon}
\renewcommand{\theta}{\vartheta}

\newcommand{\R}{{I\!\!R}}
\newcommand{\eps}{\varepsilon}
\renewcommand{\d}{\delta}
\newcommand{\C}{{\mathcal{C}}}

\begin{document}

\author{Alessio Figalli}
\address{The University of Texas at Austin,
Mathematics Department RLM 8.100,
2515 Speedway Stop C1200,
Austin TX 78712-1202 (USA)}
\email{figalli@math.utexas.edu}

\author{Enrico Valdinoci}
\address{Weierstra{\ss} Institut f\"ur Angewandte Analysis und Stochastik\\
Mohrenstra{\ss}e 39, D-10117 Berlin (Germany)}
\email{enrico.valdinoci@wias-berlin.de}

\title[Regularity and Bernstein-type results for nonlocal minimal surfaces]{Regularity and Bernstein-type results\\
 for nonlocal minimal surfaces}

\begin{abstract}
We prove that, in every dimension, Lipschitz nonlocal minimal surfaces are smooth.
Also, we extend to the nonlocal setting a famous theorem of De Giorgi \cite{DG} stating that
the validity of Bernstein's theorem in dimension $n+1$ is a consequence of the nonexistence of  $n$-dimensional singular minimal cones in $\R^n$.
\end{abstract}

\thanks{Supported by 
NSF Grant DMS-1262411 (AF) and
ERC Grant $\eps$ 277749 (EV)}

\maketitle

\section{Introduction}

Given~$n\ge1$ and~$s\in(0,1)$, we investigate the regularity of nonlocal $s$-minimal surfaces in~$\R^{n+1}$.
To begin, we recall the notion of~$s$-perimeter and~$s$-minimal surface, as introduced in~\cite{CRS}.

Given two disjoint measurable sets~$F,G\subseteq\R^{n+1}$ we consider the $s$-interaction
between them defined by
$$ L(F,G):=\iint_{F\times G}\frac{dx\,dy}{|X-Y|^{n+1+s}}.$$
Given a measurable set~$E$ and a bounded set~$\Omega\subset\R^{n+1}$,
the ``$s$-perimeter'' of the set~$E$ inside~$\Omega$ is defined as
$$ {\rm Per}_s(E,\Omega):= L(E\cap\Omega,\R^{n+1})+L(E\setminus\Omega,\Omega\setminus E).$$
We say that~$E$ is a ``$s$-minimal surface'' in~$\Omega$ if~${\rm Per}_s(E,\Omega)<+\infty$
and for any measurable set~$F\subseteq\R^{n+1}$ with~$F\setminus\Omega=E\setminus\Omega$
we have that
$$ {\rm Per}_s(E,\Omega)\le {\rm Per}_s(F,\Omega).$$
If~$E$ is a $s$-minimal surface in any ball, we simply say that~$E$ is
a $s$-minimal surface.
Namely, $s$-minimal surfaces are local minimizers of the $s$-perimeter functional.
The ``$s$-mean curvature'' of~$E$ at a point~$X\in\partial E$ is defined by
\begin{equation}\label{DE}
I[E](X):=\int_{\R^{n+1}}\frac{\chi_E(Y)-\chi_{E^c}(Y)}{|X-Y|^{n+1+s}}\,dy,\end{equation}
where~$E^c:=\R^{n+1}\setminus E$. We remark that if~$\partial E$ is $C^2$ in a neighborhood of~$X$,
then~$I[E](X)$ is well-defined in the principal value sense. On the other hand, 
while a priori 
a $s$-minimal surface $E$ may not be smooth, it is shown in~\cite{CRS} that it satisfies the equation~$I[E](X)=0$ for any~$X\in \partial E$ in a suitable viscosity sense (in particular, it satisfies
the equation in the classical sense at every point where $\partial E$ is~$C^2$).

With this notation, $s$-minimal surfaces have vanishing $s$-mean curvature, and the analogy
with the classical perimeter case is evident.
To make the analogy even stronger, we recall that, as~$s\nearrow1$, the $s$-perimeter
converges to the classical perimeter, with good geometric and functional analytic
properties, see~\cite{CV, ADM}.

{F}rom the results in~\cite{CRS, SV, BFV} it is known that boundaries of $s$-minimal
surfaces are $C^\infty$ with the possible
exception of a closed singular set of
Hausdorff dimension at most~$n-3$.

The first result of this paper shows that
Lipschitz $s$-minimal surfaces are smooth.
Notice that, in the classical case, this result is a consequence of the De Giorgi-Nash Theorem
on the H\"older regularity of solutions to uniformly elliptic equations in divergence form.
However, in this nonlocal setting it does not seem possible
to use the regularity theory
for nonlocal equations to deduce this result and
we need to employ geometric arguments instead.

\begin{theorem}\label{T1}
Let $n\ge1$ and~$E$ be a $s$-minimal surface in~$B_1\subset \R^{n+1}$. Suppose that~$\partial E\cap B_1$ is
locally Lipschitz.
Then~$\partial E\cap B_1$
is $C^\infty$.
\end{theorem}

We say that a $s$-minimal surface $E$ is a ``$s$-minimal graph'' if it
can be written as a global graph in some direction (that is, up
to a rotation,~$E=\{ (x,\tau)\in\R^n\times\R \,:\, \tau<u(x)\}$
for some function~$u:\R^n\to\R$), and it is a ``$s$-minimal cone''
if  it is a cone (that is, up to a translation, $E=tE$ for any $t>0$).
A variant of the techniques used in the proof of Theorem~\ref{T1} allows us to 
show that the validity of Bernstein's theorem in dimension $n+1$ is a consequence of the nonexistence of  $n$-dimensional singular $s$-minimal cones in $\R^n$,
thus extending to the fractional case a famous result of De Giorgi for minimal surfaces \cite{DG}:

\begin{theorem}\label{T2}
Let $E=\{ (x,\tau)\in\R^n\times\R \,:\, \tau<u(x)\}$ be a $s$-minimal graph,
and assume there there are no singular $s$-minimal cones in dimension $n$
(that is, if $\C\subset \R^{n}$ is a $s$-minimal cone, then $\C$ is a half-space).
Then~$u$ is an affine function (thus~$E$ is a half-space).
\end{theorem}

The above result combined with the non-existence of 
$s$-minimal cones in dimension $n\leq 2$ (see \cite{SV})  implies the following:

\begin{corollary}
\label{C}
Let $E=\{ (x,\tau)\in\R^n\times\R \,:\, \tau<u(x)\}$ be a $s$-minimal graph, and assume that $n \in \{1,2\}$.
Then $u$ is an affine function.
\end{corollary}
When~$n=1$ Corollary~\ref{C} is a particular case of the result in~\cite{SV}, but for~$n=2$
the result is new. 

The paper is organized as follows. Some preliminary results on Lipschitz
functions are collected in Section~\ref{S:1},
and a useful observation on the asymptotic behavior of the $s$-minimal
cones at large scale is given in Section~\ref{S:B}. Then, the proofs of Theorems~\ref{T1}
and~\ref{T2} are given in Sections~\ref{S:2} and~\ref{S:3}, respectively.

\section{Technical lemmata on Lipschitz functions}\label{S:1}

This section contains some auxiliary results
of elementary nature.

In the first lemma we show that
Lipschitz functions whose gradient is almost constant
in a suitably large set need to be uniformly close to
an affine hyperplane:

\begin{lemma}\label{Co1}
Let~$M>0$ and $\omega\in\R^n$ with $|\omega|\le M$.
Given~$\eps>0$ there exists~$\d=\delta(n,\eps,M)>0$, such
that the following holds: if~$u:B_1\rightarrow\R$ is a $M$-Lipschitz function
satisfying
$$ \Big| \{ x\in B_1 \,:\, |\nabla u(x)-\omega|\ge\d\}\Big|\le\d,$$
then~$|u(x)-u(0)-\omega\cdot x|\le\eps$ for any~$x\in B_1$.
\end{lemma}

\begin{proof}
In this proof $C$ will denote a generic constant depending only $M$, which may change from line to line.
Set $w(x):=u(x)-\omega \cdot x$.
It is immediate to check that $w$ is $(2M)$-Lipschitz and from our assumptions on $u$ we get
$$
\int_{B_1}|\nabla w(x)|\,dx=
\int_{B_1\cap\{|\nabla w|<\delta\}}|\nabla w(x)|\,dx+\int_{B_1\cap\{|\nabla w|\in [\delta,2M]\}}|\nabla w(x)|\,dx
 \leq C\d.
$$
Hence, by H\"older inequality,
$$
\|\nabla w\|_{L^{n+1}(B_1)}\leq \|\nabla w\|_{L^{1}(B_1)}^{1/(n+1)}\|\nabla w\|_{L^{\infty}(B_1)}^{n/(n+1)}
\leq C\delta^{1/(n+1)},
$$
and applying Sobolev inequality in $W^{1,n+1}(B_1)$ we deduce that there exists a constant $\ell \in \R$ such that
$$
\|w-\ell\|_{L^\infty(B_1)} \leq C\|\nabla w\|_{L^{n+1}(B_1)} \leq C\delta^{1/(n+1)}.
$$
Since
$$
\|w-w(0)\|_{L^\infty(B_1)}\leq \|w-\ell\|_{L^\infty(B_1)}+|\ell-w(0)| \leq 2\|w-\ell\|_{L^\infty(B_1)},
$$
this concludes the proof with $\eps=C\delta^{1/(n+1)}$.
\end{proof}

In the next result we observe that if a Lipschitz function
has local growth close
to the maximal one at many points, then it needs to be uniformly close to an affine map:

\begin{proposition}\label{PP-1}
Let $M>0$. Then, for any~$\eps>0$ there exists~$\mu=\mu(n,\eps,M)\in (0,M)$
such that the following holds: fix~$\sigma\in \partial B_1$,
and let $u:B_1\to \R$ be a $M$-Lipschitz function satisfying
\begin{equation}\label{0.01}
\Big|\big\{x\in B_1 \,:\,
u(x+t_k \sigma)-u(x)\ge (M-\mu)\,t_k
\big\}\Big|\ge (1-\mu)\,|B_1|
\end{equation}
for some sequence~$t_k\searrow0$.
Then $|u(x)-u(0)-M\sigma\cdot x|\le\eps$ for all~$x\in B_1$.
\end{proposition}

\begin{proof}

Up to a rotation we can assume that~$\sigma=e_1$.
Set
\begin{equation}\label{7.8}
A_k:=\big\{x\in B_1 \,:\,
u(x+t_k e_1)-u(x)\ge (M-\mu)\,t_k
\big\}\end{equation}
and
$$ A_\star :=\bigcap_{k=0}^{+\infty} \bigcup_{j=k}^{+\infty} A_j.$$
Notice that~$|A_k|\ge(1-\mu)\,|B_1|$ (thanks to~\eqref{0.01}) and
$$ \bigcap_{k=0}^{m} \bigcup_{j=k}^{+\infty} A_j\,\supseteq\, A_m.$$
Therefore, by monotone convergence,
\begin{equation}\label{0.02} |A_\star|=\lim_{m\nearrow\infty}
\left|\bigcap_{k=0}^{m} \bigcup_{j=k}^{+\infty} A_j
\right|\ge 
\lim_{m\nearrow\infty}|A_m|\ge(1-\mu)\,|B_1|.\end{equation}
Let $D\subset B_1$ denote the set of differentiability points of $u$ (recall that $D$ has full measure).
We claim that
\begin{equation}\label{0.03}
\sup_{x\in A_\star\cap D} |\nabla 
u(x)-Me_1|<\mu^{1/4}.\end{equation}
For this, we take~$x\in A_\star\cap D$. By definition of $A_\star$ there 
exists a subsequence $j_k\nearrow \infty$ such that~$x\in A_{j_k}$, thus,
by~\eqref{7.8},
$$ u(x+t_{j_k} e_1)-u(x)\ge (M-\mu)\,t_{j_k}.$$
Dividing by~$t_{j_k}>0$ and letting~$k\nearrow\infty$ we obtain (recall that $t_{j_k}\searrow 0$ as $k \nearrow \infty$)
\begin{equation}\label{9.9}
\partial_1 u(x)\ge (M-\mu).\end{equation}
As a consequence
$$ M^2\ge |\nabla u(x)|^2\ge (M-\mu)^2+\sum_{i=2}^n |\partial_i u(x)|^2,$$
which gives
$$ \sum_{i=2}^n |\partial_i u(x)|^2\le 2M\mu.$$
Also~\eqref{9.9} and the fact that~$\partial_1 u(x)\le M$ imply that
$$ -\mu\le\partial_1 u(x)-M\le 0,$$
hence~$|\partial_1 u(x)-M|\le\mu$. We conclude that
$$ |\nabla u(x)-Me_1|^2=
|\partial_1 u(x)-M|^2+\sum_{i=2}^n |\partial_i u(x)|^2\le
\mu^2+2M\mu< \sqrt\mu$$
provided~$\mu$ is sufficiently small, proving~\eqref{0.03}.

By~\eqref{0.02} and~\eqref{0.03} we deduce that
$$ \Big| \big\{ x\in B_1 \,:\,
|\nabla u(x)-Me_1|\ge\mu^{1/4}\big\}\Big|\le
|B_1\setminus A_\star|\le \mu\,|B_1|\le\mu^{1/4}.$$
Hence, if $\mu$ is small enough, we can apply 
Lemma~\ref{Co1} with $\delta=\mu^{1/4}$ to obtain the desired result.
\end{proof}

\section{A remark on flat blow-downs}\label{S:B}

First of all, we recall here the notion of blow-up and blow-down of a~$s$-minimal surface~$E$, which
will be used in the proofs of Theorems~\ref{T1} and~\ref{T2}. 

Assume that~$0\in\partial E$, define the family of sets $E_r:=E/r$, and let $E_0$ (resp. $E_\infty$) be a cluster point with respect to the $L^1_{\rm loc}$-convergence for
$E_r$ as $r \searrow 0$ (resp. $r\nearrow \infty$).

With this notation, $E_0$ is called a ``blow-up'' of $E$ (at $0$), while $E_\infty$ is called a ``blow-down''.
By \cite[Theorem 9.2]{CRS} we know that~both $E_0$ and $E_\infty$ are
$s$-minimal cones passing through the origin.

In the proof of Theorem~\ref{T2} we will use the following observation:

\begin{lemma}\label{BL}
If~$E_\infty$ is a half-space, then~$E=E_\infty$.
\end{lemma}

\begin{proof} Up to a rotation we can assume that~$E_\infty=\{(x,\tau)\in\R^n\times\R \,:\, \tau\le 0\}$.
Let $r_k \nearrow \infty$ be a sequence such that $E_{r_k} \to E_\infty$, and let 
$\eps_0$ be the universal flatness parameter in~\cite[Theorem~6.1]{CRS}.
By the uniform density estimates for $s$-minimal surfaces (see \cite[Theorem 4.1]{CRS}), we have
that $(\partial E_{r_k}) \cap B_1 \to (\partial E_\infty)\cap B_1$ in the Hausdorff distance as $r_k \nearrow \infty$.
Hence, for $r_k$ sufficiently large
$E_{r_k}\cap B_1$ lies in an~$\eps_0$-neighborhood of~$E_\infty$,
and \cite[Theorem~6.1]{CRS} yields that~$(\partial E_{r_k})\cap B_{1/2}$ is a $C^{1,\alpha}$-graph parameterized by a function $u_{r_k}:
B_1 \to\R$, with~$\|u_{r_k}\|_{C^{1,\alpha}(B_{1/2})}\le C$
for some universal constants~$\alpha\in(0,1)$
and~$C>0$.

Scaling back, we deduce that~$(\partial E)\cap B_{r_k/2}$
coincides with the graph of
a function $u$ which satisfies $u(x)=r_k u_{r_k} (x/r_k)$
and $u(0)=0$ (since $0 \in \partial E$).
Hence
$$r_k^\alpha [\nabla u]_{C^{\alpha}(B_{r_k/2})}= [\nabla u_{r_k}]_{C^{\alpha}(B_{1/2})}\le C,$$
and by letting~$r_k\nearrow \infty$ we see that~$\nabla u$ is constant. Thus~$u$
is a linear function, which implies that $E$ is a half-space.
Since $0 \in \partial E$ it is immediate to check that $E=E_r$ for all $r>0$,
therefore (by letting $r \nearrow \infty$) $E=E_\infty$ as desired.
\end{proof}

\section{Proof of Theorem~\ref{T1}}\label{S:2}
By \cite[Theorem~6.1]{CRS} and \cite[Theorem 5]{BFV}, there exists $\eps_0>0$ such that,
if $B_r(X)\subset B_1$ and $(\partial E)\cap B_r(X)$ lies in a slab of height~$2\eps_0r$,
then $(\partial E)\cap B_{r/2}(X)$ is $C^\infty$.
Hence we only need to show that, for any $X \in B_1$, there exists a radius $r<1-|X|$ such that
$(\partial E)\cap B_r(X)$ lies in a slab of height~$2\eps_0r$.

So, we fix $X_0 \in B_1$, we suppose (up to a change of coordinates)
that $\partial E$ is a Lipschitz graph in the $e_{n+1}$-direction in a neighborhood of $X$,
and we assume by contradiction that, for any $r>0$ small, 
$(\partial E)\cap B_r(X_0)$ is never trapped inside a slab of height~$2\eps_0r$.

Up to translate the system of coordinate we can assume that $X_0=0$,
and we consider a  blow-up~$E_0$ of~$E$ (recall the notation
of blow-ups presented in Section \ref{S:B}).
By \cite[Theorem 9.2]{CRS} we know that~$E_0$
is a Lipschitz $s$-minimal cone passing through the origin,
and, by uniform density estimates for $s$-minimal surfaces (see \cite[Theorem 4.1]{CRS}),
it is immediate to check that 
$(\partial E_0)\cap B_R$ is never trapped inside a slab of height~$2\eps_0R$ for any $R>0$.

Now, up to a standard ``dimension reduction argument'' (see \cite[Theorem 10.3]{CRS})
we can ``remove'' all the singular points
of $\partial E_0$ except the origin and we end up\footnote{We notice that,
since~$\partial E_0$ is a Lipschitz graph, one can perform the
dimension reduction argument without changing system of coordinates. Therefore,
after a finite number of blow-ups, we still end up with a Lipschitz graph.}
with the following situation:
$E_0$ is a Lipschitz
cone passing through the origin,
\begin{equation}\label{C2}
{\mbox{$(\partial E_0)\cap B_1$ does not
lie in a slab of height~$\eps_0$,}}\end{equation}
and $\partial E_0$ is a Lipschitz graph in the $e_{n+1}$-direction which is 
smooth outside the origin, that is
$$ E_0=\big\{ X=(x,\tau)\in\R^n\times\R \,:\,
\tau<u(x)\big\},$$
\begin{equation}\label{Cinf}
u(0)=0,\qquad
u\in C^2(\R^n\setminus\{0\}),\qquad  |\nabla u(x)|\le M\quad \text{for any~$x\ne0$}.
\end{equation} 
To be precise, the dimension reduction argument in~\cite{CRS}
gives that $u\in C^{1,\alpha}(\R^n\setminus\{0\})$, and by~\cite[Theorem 1]{BFV} we obtain that
$u\in C^\infty(\R^n\setminus\{0\})$.
Of course we can take $M>0$ to be the smallest possible (i.e., $M$ is
the optimal Lipschitz constant of~$u$).

%% To find a contradiction it suffices to show that
%% \begin{equation}\label{C1}
%% {\mbox{$(\partial E_0)\cap B_1$ is a $C^{1,\alpha}$-graph.}}
%% \end{equation}
%% Indeed, since $E_0$ is a cone, this would
%% imply that $E_0$ is a half-space, which
%% is incompatible with \eqref{C2}.

Take~$\mu_0:=\mu(n,\eps_0/2,M)$ as
in Proposition~\ref{PP-1}. Then it follows from~\eqref{C2}
that~\eqref{0.01} cannot hold true.
Hence, for any~$\sigma\in \partial B_1$
there exists~$t_\sigma>0$
such that
\begin{equation}\label{09} \Big|\big\{x\in B_1 \,:\,
u(x+t \sigma)-u(x)< (M-\mu_0)\,t
\big\}\Big|\ge\mu_0\,|B_1|\end{equation}
for all~$t\in(0,t_\sigma)$.
Now we take~$w_0\in C^\infty(\R,[0,1])$, with~$w_0(t)=0$ for any~$t\in(-\infty,1/4]\cup[3/4,+\infty)$
and~$w_0(t)=1$ for any~$t\in[2/5,3/5]$. We set~$w(x)=w_0(|x|)$ and we observe that
\begin{equation}\label{12}
{\mbox{$w(x)=1$ for any~$x\in B_{3/5}\setminus B_{2/5}$.}}
\end{equation}
Our goal is to
show that there exists a constant~$\theta>0$ such that
\begin{equation}\label{10}
u(x+t\sigma)\le u(x)+M t-\theta t w(x)\qquad \forall\,x\in B_1,\,t\in(0,t_\sigma),\,\sigma\in \partial B_1.
\end{equation}
Before proving~\eqref{10} we observe 
that, once~\eqref{10}
is established, we easily reach a contradiction and complete the proof
of Theorem~\ref{T1}. Indeed, letting $t \searrow0$ in~\eqref{10} and using~\eqref{12}
we deduce that
$$ \nabla u(x)\cdot\sigma\le M -\theta w(x)=M-\theta\qquad \forall\,x\in B_{3/5}\setminus B_{2/5},\,\sigma\in \partial B_1,$$
hence
$$
|\nabla u(x)|\le M-\theta\qquad \forall\,x\in B_{3/5}\setminus B_{2/5}
$$ by the arbitrariness of $\sigma$.
Since~$\nabla u$ is homogeneous of degree zero it follows that~$|\nabla u(x)|\le M-\theta$
for any~$x\ne0$, which contradicts our assumption that~$M$ was the optimal Lipschitz
constant of~$u$.
So, it only remains to prove~\eqref{10}.

For this we consider the surfaces
\begin{eqnarray*}
&& F:=\{ (x,\tau)\in\R^n\times\R \,:\, \tau<u(x+t\sigma)\}\\
{\mbox{and }}&& G_{\theta,\alpha}:=
\{ (x,\tau)\in\R^n\times\R \,:\, \tau<u(x)+(M-\theta w(x))t+\alpha\}.
\end{eqnarray*}
Notice that \eqref{10} is equivalent to 
\begin{equation}\label{theta0}
F \subseteq G_{\theta,0}.
\end{equation}
To prove \eqref{theta0} we first observe that
$$ u(x+t\sigma)\le u(x)+Mt\le u(x)+(M-\theta w(x))t+Mt$$
provided $\theta \leq M$, 
thus~$F\subseteq G_{\theta,\alpha}$ for any~$\alpha>Mt$. Now we reduce $\alpha$
till we find a critical~$\alpha_0$ for which~$G_{\theta,\alpha_0}$
touches~$F$ from above. We claim that
\begin{equation}\label{theta0.1}
\alpha_0\le0.\end{equation}
Suppose by contradiction that~$\alpha_0>0$. Since $u$ is $M$-Lipschitz
we have that
$$
u(x+t\sigma) \leq u(x)+Mt <u(x)+Mt+\alpha_0.
$$
which implies that $G_{\theta,\alpha_0}$
and~$F$ can only touch at a some point~$X_0=(x_0,t_0)\in\R^n\times \R$
with $x_0 \in {\rm supp}(w)\subset B_{3/4}\setminus B_{1/4}$.
Hence, it is easy to see (by compactness) that a contact point $X_0=(x_0,t_0)$ exists,
and since  $x_0 \in B_{3/4}\setminus B_{1/4}$ we have that 
both sets are uniformly $C^2$ near~$X_0$, so the 
$s$-mean curvature operators~$I[F]$ and~$I[G_{\theta,\alpha_0}]$ (recall \eqref{DE})
may be computed at~$X_0$ in the classical sense.

Since~$F$ and~$G_{0,\alpha_0}$ are $s$-minimal surfaces, we have that
\begin{equation}\label{co1}
I[F](X_0)=0=I[G_{0,\alpha_0}](X_0).
\end{equation}
Also, since~$G_{\theta,\alpha_0}$ is a $C^2$-diffeomorphism of~$G_{0,\alpha_0}$
of size~$\theta t$, and $G_{\theta,\alpha_0}$ is uniformly~$C^2$
in a neighborhood of~$X_0$, we have that
\begin{equation}\label{co2}
\big|I[G_{\theta,\alpha_0}](X_0)\big|\le C\theta t,
\end{equation}
for some universal constant~$C>0$.
Furthermore, since~$F\subseteq G_{\theta,\alpha_0}$ we have that
\begin{equation}\label{219}
\chi_{G_{\theta,\alpha_0}}-\chi_{G_{\theta,\alpha_0}^c}-\chi_F+\chi_{F^c}=2\chi_{G_{\theta,\alpha_0}\setminus F}.\end{equation}
Now we define
$$ Z:=
\big\{x\in B_1\,:\,
u(x+t \sigma)-u(x)< (M-\mu_0)\,t
\big\}\subset\R^n$$
and
$$ W:=\big\{ (x,\tau)\in\R^n\times \R \,:\, x\in Z {\mbox{ and }} u(x+t\sigma)<\tau<u(x+t\sigma)+\mu_0t/2\big\}
\subset\R^{n+1}.$$
We remark that~$|Z|\ge\mu_0\,|B_1|$ thanks to~\eqref{09},
therefore~$|W|\ge \mu_0^2\,t\,|B_1|/2$. 
(Notice that, by abuse of notation, we are using $|\cdot|$ to denote both the Lebesgue measure in $\R^n$ and $\R^{n+1}$.)

We claim that
\begin{equation}\label{99}
(G_{\theta,\alpha_0}\setminus F)\supseteq W
\end{equation}
provided $\theta$ is sufficiently small.
Indeed, let~$(x,\tau)\in W$. Then~$x\in Z$ and~$u(x+t\sigma)<\tau<u(x+t\sigma)+\mu_0t/2$.
This says that~$(x,\tau)\not\in F$ and
\begin{eqnarray*}
&& \tau<u(x+t\sigma)-\frac{\mu_0t}{2}<
u(x)+(M-\mu_0)\,t+\frac{\mu_0t}{2}
\\ &&\qquad\le
u(x)+Mt-\theta t\le u(x)+Mt-\theta w(x)t\end{eqnarray*}
provided~$\theta\in (0,\mu_0/4)$. This shows that~$(x,\tau)\in G_{\theta,\alpha_0}$
proving~\eqref{99}.

Since by construction $Z\subseteq B_1\subset\R^n$
and $u$ is $M$-Lipschitz with $u(0)=0$, we deduce that~$W\subset B_1\times [-1-2M,1+2M]$,
which implies that 
$$ \sup_{Y\in W} |X_0-Y|\le C_M$$
for some~$C_M>0$, and (by~\eqref{99}) that
$$ \big|(G_{\theta,\alpha_0}\setminus F)\cap W\big|= |W|\ge \frac{\mu_0^2\,t\,|B_1|}{2}.$$
{F}rom this, \eqref{DE}, and~\eqref{219}, we conclude that
\begin{eqnarray*}
I[G_{\theta,\alpha_0}](X_0) -I[F](X_0) &=&
\int_{\R^{n+1}}\frac{2\chi_{G_{\theta,\alpha_0}\setminus F}(Y)}{|X_0-Y|^{n+1+s}}\,dy\\
&\ge&
\int_{W}\frac{2\chi_{G_{\theta,\alpha_0}\setminus F}(Y)}{|X_0-Y|^{n+1+s}}\,dy\\
&\ge& 2 C_M^{-n-1-s}
\int_{W} \chi_{G_{\theta,\alpha_0}\setminus F}(Y)\,dy\\
&=& 2 C_M^{-n-1-s}\,|W|\\
&\ge& 
C_M^{-n-1-s}\mu_0^2\,t\,|B_1|.
\end{eqnarray*}
Hence, 
combining~\eqref{co1} and~\eqref{co2} we get
\begin{equation}\label{3BIS} C\theta t\ge
I[G_{\theta,\alpha_0}](X_0)=
I[G_{\theta,\alpha_0}](X_0)-
I[F](X_0)\ge
C_M^{-n-1-s}\mu_0^2\,t\,|B_1|,\end{equation}
which is a contradiction if~$\theta$ is sufficiently small.
This contradiction proves~\eqref{theta0.1},
that in turn implies~\eqref{theta0} and so~\eqref{10}.
This concludes the proof of Theorem~\ref{T1}.

\section{Proof of Theorem~\ref{T2}}\label{S:3}
Let $E_\infty$ be a blow-down of $E$, that is a cluster point for $E_r:=E/r$
as $r \nearrow \infty$.
In this way we get a $s$-minimal cone, and the assumption that 
no singular $s$-minimal cones exist in dimension $n$ combined with a standard dimension reduction argument
implies that $E_\infty$ can only be singular at the origin.

Also, because $\partial E$ was a graph, $E_\infty$ is an hypograph in $\R^{n+1}$, that is
\begin{equation}
\label{epigr}
(x,\tau) \in E_\infty \quad \Longrightarrow \quad (x,\tau - t) \in E_\infty \qquad \forall\, t \geq 0. 
\end{equation}
Now we show that $E_\infty$ is in fact a graph (and not only
an hypograph). 
For this, suppose by contradiction
that
there exists $\tau_\infty>0$ such that $\partial E_\infty$ touches $\partial E_\infty+\tau_\infty e_{n+1}$ at some point.
Then,
by the strong maximum principle\footnote{A simple and direct way to see the strong maximum
principle is to use that $E_\infty$ and $E_\infty+\tau_\infty e_{n+1}$ are smooth cones outside the origin. So, if they touch, we can find a contact point $X_0 \neq 0$, and by computing the operator $I$
defined in \eqref{DE} at $X_0$ for both surfaces, since both $E_\infty$ and $E_\infty+\tau_\infty e_{n+1}$
are $s$-minimal and $E_\infty\subset E_\infty+\tau_\infty e_{n+1}$
we get
$$
0=I[E_\infty+\tau_\infty e_{n+1}](X_0)-I[E_\infty](X_0)=
\int_{\R^{n+1}}\frac{2\chi_{E_\infty+\tau_\infty e_{n+1}\setminus E_\infty}(Y)}{|X_0-Y|^{n+1+s}}\,dy,
$$
which implies that $E_\infty=E_\infty+\tau_\infty e_{n+1}$, as desired.}
we get $E_\infty=E_\infty+\tau_\infty e_{n+1}$,
from which (iterating this equality) we get $E_\infty=E_\infty+k\tau_\infty e_{n+1}$ 
for any $k \in \mathbb N$. This fact
combined with \eqref{epigr} implies that 
$$
E_\infty=\C\times \R,
$$
where $\C$ is a $s$-minimal cone in $\R^n$. Hence it follows by our assumption that
$\C$ is a half-space,
and Lemma~\ref{BL}
gives that $E=\C\times \R$
which is in contradiction with the fact that $E$ was a graph.

Hence we have shown that $\partial E_\infty$ and $\partial E_\infty+\tau e_{n+1}$
never touch for any $\tau>0$, which implies that
$\partial E_\infty$ is the graph of a function $u_\infty:\R^n\to \R$.
In addition, since $E_\infty$ is smooth outside the origin, so is $u_\infty$.

Now, as in the proof of Theorem~\ref{T1}, we consider
a bump function~$w_0\in C^\infty(\R,[0,1])$, with~$w_0(t)=0$ for any~$t\in(-\infty,1/4]\cup[3/4,+\infty)$
and~$w_0(t)=1$ for any~$t\in[2/5,3/5]$, and we define~$w(x)=w_0(|x|)$.
Then, we fix $\sigma \in \partial B_1$ and  consider the family of sets
$$
F_t:=\bigl\{(x,\tau)\,:\, \tau \leq u_\infty\bigl(x+t\theta w(x)\sigma\bigr)-t\bigr\},
$$
where $t\in[0,1]$
and $\theta >0$.
By compactness 
we see that, if $\theta$ is sufficiently small, then $F_1\subseteq E_\infty$.
Let $t_0 \in [0,1]$ be the smallest $t$ for which $F_t \subseteq E_\infty$,
and assume by contradiction that $t_0>0$.
Since~$E_\infty$ is a graph, we see that $F_{t_0}$ can only touch $E_\infty$ from below at some point $X_0=(x_0,t_0)$ with $x_0 \in {\rm supp}(w)\subset B_{3/4}\setminus B_{1/4}$.
Hence, it is easy to see (by compactness) that a contact point $X_0=(x_0,t_0)$ exists,
and since  $x_0 \in B_{3/4}\setminus B_{1/4}$ we have that 
both sets are smooth near $X_0$.

Therefore
we can easily adapt the arguments provided in~\eqref{co1}--\eqref{3BIS} as follows:
First, by the~$s$-minimality of $E_\infty$ we have that~$I[E_\infty](X_0)=0=I[F_0](X_0)$.
Also, since~$F_{t_0}$ is a $C^2$-diffeomorphism of~$F_0$ of size~$\theta t_0$
and~$F_{0}$ is uniformly $C^2$ in a neighborhood of~$X_0$, we have that
\begin{equation}\label{P0}
\bigl| I[F_{t_0}](X_0) \bigr| \leq C\theta t_0.
\end{equation}
On the other hand, since the graph of $u_\infty$ is uniformly Lipschitz in a non-trivial fraction of points
(just pick a point where the tangent space to $\partial E_\infty$ is not vertical and consider a small neighborhood of this point)
we see that $\partial F_{t_0}$ and $\partial E_\infty$ lie at distance $\geq ct_0$ on a non-trivial fraction of points,
therefore
$$ \big| (E_\infty\setminus F_{t_0}) \cap B_1\big|\ge c_0 t_0$$
for some~$c_0>0$. Hence, arguing as in the proof of Theorem \ref{T1} we get
$$ \bigl| I[F_{t_0}](X_0)\bigr|=
\bigl| I[F_{t_0}](X_0)-I[E_\infty](X_0) \bigr| \geq c't_0
$$
for some~$c'>0$, which is in contradiction with~\eqref{P0}
if $\theta$ was chosen sufficiently small.

This proves that $t_0=0$, which implies that $F_t\subseteq E_\infty$ for any $t \in (0,1)$, or equivalently
$$ \frac{u_\infty(x+t\theta w(x)\sigma)-u_\infty(x)}{t}\le 1
\qquad \forall \,t \in (0,1).$$
Hence, letting $t \searrow 0$ we obtain
$$
\theta w(x)\nabla u_\infty(x)\cdot \sigma\le1\qquad \forall\,x\in\R^n\setminus\{0\},\,\sigma\in\partial B_1,
$$
which combined with the fact that~$w=1$ in~$B_{3/5}\setminus B_{2/5}$ and $\sigma \in \partial B_1$
is arbitrary implies
$$
|\nabla u_\infty(x)|\le1/\theta\qquad \forall\,x\in
B_{3/5}\setminus B_{2/5}.
$$

Since $u_\infty$ is $1$-homogeneous we deduce that $u_\infty$ is globally Lipschitz.
So by Theorem~\ref{T1} it is smooth also at the origin, hence (being a cone) $E_\infty$ a half-space.
Using again Lemma~\ref{BL}
we deduce that~$E$ is a half-space as well, 
concluding the proof of Theorem~\ref{T2}.

\end{document}